\documentclass[12pt,twoside]{amsart}
\usepackage{amsmath,amsthm,amsmath,amsrefs}

\usepackage{calc}
\usepackage{setspace}
\usepackage{amsbsy,amsmath,amsfonts,amsthm,amssymb,color}
\usepackage[margin=1in]{geometry}
\numberwithin{equation}{section}
\theoremstyle{plain}
\newtheorem{theorem}{Theorem}[section] 
\newtheorem{lemma}[theorem]{Lemma}
\newtheorem{definition}[theorem]{Definition}

\newtheorem{proposition}[theorem]{Proposition}

\theoremstyle{definition}

\usepackage{mathrsfs}

\newcommand{\bN}{\mathbb{N}}

\newcommand{\cS}{\mathcal{S}}
\newcommand{\cB}{\mathcal{B}}

\newcommand{\cK}{\mathcal{K}}

\newcommand{\la}{\langle}
\newcommand{\ra}{\rangle}

\newcommand{\cH}{\mathcal{H}}

\newcommand{\bofh}{\cB(\cH)}
\newcommand{\cP}{\mathcal{P}}
\newcommand{\cM}{\mathcal{M}}

\newcommand{\bC}{\mathbb{C}}
\newcommand{\bR}{\mathbb{R}}

\newcommand{\bZ}{\mathbb{Z}}

\newcommand{\spn}{\text{span }}

\usepackage{tikz}
\usepackage{tikz-cd}

\title[Bipartite matrix-valued tensor product correlations]{Bipartite matrix-valued tensor product correlations that are not finitely representable}

\author{Samuel J. Harris}
\address{University of Waterloo
Department of Pure Mathematics 
200 University Ave. W.
Waterloo, Ontario 
Canada N2L 3G1}
\email{sj2harri@uwaterloo.ca}

\begin{document}
\begin{abstract}
We consider the matrix-valued generalizations of bipartite tensor product quantum correlations and bipartite infinite-dimensional tensor product quantum correlations, respectively. These sets are denoted by $C_q^{(n)}(m,k)$ and $C_{qs}^{(n)}(m,k)$, respectively, where $m$ is the number of inputs, $k$ is the number of outputs, and $n$ is the matrix size. We show that, for any $m,k \geq 2$ with $(m,k) \neq (2,2)$, there is an $n \leq 4$ for which we have the separation $C_q^{(n)}(m,k) \neq C_{qs}^{(n)}(m,k)$.
\end{abstract}

\maketitle

\section{Introduction}

The study of quantum bipartite correlations has been of great importance to understanding the nature of entanglement. These correlations describe the probability of obtaining certain outputs in a separated system, assuming that certain inputs were given. We assume that we are working in a finite input, finite output system. For such a correlation, we will denote by $p(a,b|x,y)$ the probability that Alice and Bob output $a$ and $b$, respectively, given that they had inputs $x$ and $y$, respectively. The different models for these probability distributions in $m$ inputs and $k$ outputs have been studied extensively in recent years. The models that we will focus on in this paper are the quantum (finite-dimensional tensor product) correlations, the quantum spatial correlations, and the quantum approximate correlations. (See \cite{fritz,junge+,tsirelson} for more information on these correlation sets.) Recall that a projection-valued measure (PVM) with $k$ outputs on a Hilbert space $\cH$ is a set of orthogonal projections $\{P_i\}_{i=1}^k$ on $\cH$ such that $\sum_{i=1}^k P_i=I_{\cH}$. The set of quantum (finite-dimensional tensor product) correlations, $C_q(m,k)$, is the set of all correlations of the form 
$$p(a,b|x,y)=\la (E_{a,x} \otimes F_{b,y})\zeta,\zeta \ra,$$
where there are finite-dimensional Hilbert spaces $\cH_A$ and $\cH_B$ such that $\zeta \in \cH_A \otimes \cH_B$ is a unit vector; for each $1 \leq x \leq m$, $\{E_{a,x}\}_{a=1}^k$ is a PVM on $\cH_A$; and for each $1 \leq y \leq m$, $\{F_{b,y}\}_{b=1}^k$ is a PVM on $\cH_B$. The set $C_{qs}(m,k)$ of quantum spatial correlations is defined in the same manner, except that we drop the assumption that $\cH_A$ and $\cH_B$ are finite-dimensional. The set of quantum approximate correlations, $C_{qa}(m,k)$, is defined to be the closure of the set $C_q(m,k) \subseteq \bR^{m^2k^2}$. It is known that $C_{qa}(m,k)$ is also the closure of $C_{qs}(m,k)$ \cite{SW}. In particular,
$$C_q(m,k) \subseteq C_{qs}(m,k) \subseteq C_{qa}(m,k),$$
and each of these sets is convex. Until recently, it was unknown whether either of these containments were strict. W.~Slofstra \cite{slofstra17} first showed that there exist large values of $m$ and $k$ for which $C_{qs}(m,k) \neq C_{qa}(m,k)$. Subsequently, S.-J.~Kim, V.~Paulsen and C.~Schafhauser \cite{KSP} gave an example of a synchronous correlation in $C_{qa}(m,k) \setminus C_{qs}(m,k)$. The smallest known example was found by K.~Dykema, V.~Paulsen and J.~Prakash \cite{DPP}, who showed that $C_{qs}(5,2) \neq C_{qa}(5,2)$. On the other hand, the question of whether $C_q(m,k)=C_{qs}(m,k)$ remained open in general, until very recently, when A.~Codalangelo and J.~Stark \cite{CS} proved that $C_q(5,3) \neq C_{qs}(5,3)$.

Along a similar line to the sets $C_t(m,k)$ for $t \in \{q,qs,qa\}$, a generalization of these sets to so-called ``matrix-valued correlations" was introduced in \cite{fritz,junge+}. Part of the motivation for introducing these sets was that they give a useful description of Connes' embedding problem in operator algebras \cite{connes}. These sets led to N.~Ozawa proving in \cite{ozawa} that Connes' embedding problem is, in fact, equivalent to determining whether $C_{qa}(m,k)$ is equal to the set $C_{qc}(m,k)$ of quantum commuting correlations for all $m,k \geq 2$ with $(m,k) \neq (2,2)$.

These matrix-valued correlations also arise naturally as strategies for what are known as bipartite steering games (see \cite{fritz} for more information on these games). Moreover, in recent work of L.~Gao, the author and M.~Junge \cite{gaoharrisjunge}, it is shown that one can find separations of the matrix-valued correlation sets in the $qa$ and $qs$ models for smaller input and output sizes than those that are known in the case of scalar-valued correlations.

\begin{definition}
Let $m,k \geq 2$ and $n \in \bN$. We define the set of $M_n$-valued quantum correlations to be
$$C_q^{(n)}(m,k)=\left\{ \left( W^*(E_{a,x} \otimes F_{b,y})W \right)_{a,b,x,y} \right\},$$
where, for each $x \in \{1,...,m\}$ and $y \in \{1,...,m\}$, $\{E_{a,x}\}_{a=1}^k$ and $\{F_{b,y}\}_{b=1}^k$ are PVMs on finite-dimensional Hilbert spaces $\cH_A$ and $\cH_B$ respectively, and $W:\bC^n \to \cH_A \otimes \cH_B$ is an isometry.

We define the set of $M_n$-valued quantum spatial correlations $C_{qs}^{(n)}(m,k)$ in the same way, only that we no longer require that $\cH_A$ and $\cH_B$ be finite-dimensional.
\end{definition}

Note that both of $C_q^{(n)}(m,k)$ and $C_{qs}^{(n)}(m,k)$ are convex subsets of $(M_n)^{m^2k^2}$ \cite{fritz,junge+}.  We will often denote $P(a,b|x,y)=W^*(E_{a,x} \otimes F_{b,y})W$, for fixed $a,b \in \{1,...,k\}$ and $x,y \in \{1,...,m\}$. We will also use the notation for ``matrix-valued" marginal distributions; i.e., we will define
\begin{align} P_A(a|x)&=\sum_{b=1}^k P(a,b|x,y)=W^*(E_{a,x} \otimes I)W, \text{ and} \label{alicemarginal}\\
P_B(b|y)&=\sum_{a=1}^k P(a,b|x,y)=W^*(I \otimes F_{b,y})W. \label{bobmarginal}
\end{align}
The marginal distribution $P_A(a|x)$ does not depend on the choice of $y$ in the sum; similarly, $P_B(b|y)$ does not depend on the choice of $x$.

\begin{definition}
We define $C_{qa}^{(n)}(m,k)$ to be the closure of $C_q^{(n)}(m,k)$ in $(M_n)^{m^2k^2}$.
\end{definition}

It follows from a theorem of J.~Bunce and N.~Salinas \cite{buncesalinas} that $C_{qa}^{(n)}(m,k)$ is also the closure of $C_{qs}^{(n)}(m,k)$.  We also have the inclusions
$$C_q^{(n)}(m,k) \subseteq C_{qs}^{(n)}(m,k) \subseteq C_{qa}^{(n)}(m,k).$$

In \cite{gaoharrisjunge}, it was shown that $C_{qs}^{(5)}(3,2) \neq C_{qa}^{(5)}(3,2)$ and $C_{qs}^{(13)}(2,3) \neq C_{qa}^{(13)}(2,3)$. In this paper, we prove analogous separations between the $q$ and $qs$ models. 

\begin{theorem}
\label{theorem: matrix version of q always different from qs}
For any $m,k \in \bN$ with $m,k \geq 2$, $(m,k) \neq (2,2)$, there is $n \leq 4$ such that $C_q^{(n)}(m,k) \neq C_{qs}^{(n)}(m,k)$.
\end{theorem}

In particular, in Theorem \ref{theorem: 3 inputs 2 outputs} we prove that $C_q^{(3)}(3,2) \neq C_{qs}^{(3)}(3,2)$, and in Theorem \ref{theorem: 2 inputs 3 outputs} we prove that $C_q^{(4)}(2,3) \neq C_{qs}^{(4)}(2,3)$. Our methods draw on using an explicit tensor product of unitary representations of the group $\bZ_2*\bZ$ and associated behaviour that cannot be witnessed on finite-dimensional Hilbert spaces. We use certain facts about group embeddings regarding free groups to translate these representations into the context of matrix-valued correlations. The interested reader can see \cite{delaharpe} for more information on these group embeddings.

In the case when $n=1$, the only known separation of $C_q(m,k)$ from $C_{qs}(m,k)$ is the recent result of A.~Codalangelo and J.~Stark, which says that $C_q(5,3) \neq C_{qs}(5,3)$. On the other hand, for smaller input and output sets, it is not known whether $C_q(m,k)=C_{qs}(m,k)$. It is widely thought that $C_q(3,2) \neq C_{qs}(3,2)$; indeed, it was conjectured by K.F.~P\'{a}l and T.~V\'{e}rtesi \cite{palvertesi} that the I3322 inequality should have a maximal violation in $C_{qs}(3,2)$ with no such violation in $C_q(3,2)$. If this conjecture were true, then it would imply that $C_q(3,2) \neq C_{qs}(3,2)$. Nonetheless, determining whether $C_q(m,k) \neq C_{qs}(m,k)$ for input and output sets smaller than the example of \cite{CS} remains open. However, if we allow for matrix-valued correlations, then by  Theorem \ref{theorem: matrix version of q always different from qs}, an analogue of the separation $C_q(m,k) \neq C_{qs}(m,k)$ will hold for any pair $(m,k)$ with $m,k \geq 2$ and $(m,k) \neq (2,2)$.

The paper is organized as follows. In Section \S2, we will prove Theorem \ref{theorem: 3 inputs 2 outputs}, which states that $C_q^{(3)}(3,2) \neq C_{qs}^{(3)}(3,2)$. In Section \S3, we will prove Theorem \ref{theorem: 2 inputs 3 outputs}, which states that, for some $n \in \{2,3,4\}$, we have $C_q^{(n)}(2,3) \neq C_{qs}^{(n)}(2,3)$. To work towards proving Theorems \ref{theorem: 3 inputs 2 outputs} and \ref{theorem: 2 inputs 3 outputs}, we need two results from \cite{gaoharrisjunge}.

\begin{theorem}
\label{theorem: spatial infinite-dimensional unitaries}
Let $\cH$ be a Hilbert space. The following are equivalent:
\begin{enumerate}
\item
$\cH$ is infinite-dimensional;
\item
There exist unitaries $T_1,T_2,U_1,U_2$ on $\cH$ and unit vectors $\zeta_1,\zeta_2 \in \cH \otimes \cH$ such that
\begin{align}
\la (U_1 \otimes U_2)\zeta_1,\zeta_2 \ra&=\la (U_1 \otimes U_2)\zeta_1,\zeta_1 \ra=\frac{1}{\sqrt{2}} \label{GHJspatialnonfiniteexample} \\
\la (T_1 \otimes I) \zeta_1,\zeta_1 \ra&=\la (I \otimes T_2)\zeta_1,\zeta_1 \ra=1 \\
\la (T_1 \otimes I) \zeta_2,\zeta_2 \ra&=\la (I \otimes T_2) \zeta_2,\zeta_2 \ra=-1. \label{GHJmarginals}
\end{align}
\end{enumerate}
\end{theorem}

In Theorem \ref{theorem: spatial infinite-dimensional unitaries}, the condition that $\cH$ be infinite-dimensional can easily be replaced with Hilbert spaces $\cH_A$ and $\cH_B$ that are infinite-dimensional, and the tensor product $\cH \otimes \cH$ can be replaced with $\cH_A \otimes \cH_B$. Indeed, if we start with Hilbert spaces $\cH_A$ and $\cH_B$ with $\dim(\cH_A)>\dim(\cH_B)$, then we may enlarge $\cH_B$ to have the same dimension as $\cH_A$ and extend the unitaries on $\cH_B^{\perp}$ by defining them to be the identity on $\cH_B^{\perp}$, and equations (\ref{GHJspatialnonfiniteexample})--(\ref{GHJmarginals}) will still hold.

The following explicit representation satisfying Theorem \ref{theorem: spatial infinite-dimensional unitaries} in the case of $\cH=\ell^2(\bZ)$ will be helpful for our purposes (see \cite{gaoharrisjunge}).

\begin{proposition}
\label{proposition: spatial rep of tensor of z2*z}
Let $\cH=\ell^2(\bZ)$, and let $T,U$ be unitaries on $\cH$ given by
\begin{equation}
Te_j=\begin{cases} e_j & j<0 \\ -e_j & j \geq 0 \end{cases}  \, \, \, \, \, \text{ and } \, \, \, \, \, Ue_j=e_{j+1} \label{GHJunitaries}.
\end{equation}
Define unit vectors $\zeta_1,\zeta_2$ in $\cH \otimes \cH$ via

\begin{equation}
\zeta_1=\sum_{j<0} (\sqrt{2})^j e_j \otimes e_j \, \, \, \, \, \text{ and } \, \, \, \, \, \zeta_2=e_0 \otimes e_0. \label{GHJvectors}
\end{equation}
Then setting $T_1=T_2=T$ and $U_1=U_2=U$ yields equations (\ref{GHJspatialnonfiniteexample})--(\ref{GHJmarginals}) of Theorem \ref{theorem: spatial infinite-dimensional unitaries}.
\end{proposition}

Notice that $T$ is a self-adjoint unitary, so that $T^2=I$. Thus, the unitaries in Proposition \ref{proposition: spatial rep of tensor of z2*z} arise from a unital $*$-homomorphism $\pi:C^*(\bZ_2*\bZ) \to \cB(\ell^2(\bZ))$ given by $\pi(\sigma)=T$ and $\pi(g)=U$, where $\bZ_2*\bZ$ is the free product of the two-element group $\bZ_2$ and the group of integers; $\sigma$ is the generator of the copy of $\bZ_2$ in $\bZ_2*\bZ$, and $g$ is a generator of $\bZ$ in $\bZ_2*\bZ$.

\section{Three Inputs and Two Outputs}

In this section, we will exhibit a correlation in $C_{qs}^{(3)}(3,2)$ that is not in $C_q^{(3)}(3,2)$.

\begin{theorem}
\label{theorem: 3 inputs 2 outputs}
There exists an element $P=(P(a,b|x,y))_{a,b,x,y}$ in $C_{qs}^{(3)}(3,2)$ such that
\begin{align}
P(2,2|2,2)-P(1,2|2,2)-P(2,1|2,2)+P(1,1|2,2)&=\begin{pmatrix} 0 & 0 & \frac{1}{\sqrt{2}} \\ 0 & 0 & \frac{1}{\sqrt{2}} \\ \frac{1}{\sqrt{2}} & \frac{1}{\sqrt{2}} & 0 \end{pmatrix} \label{secondobservables} \\
P(2,2|3,3)-P(1,2|3,3)-P(2,1|3,3)+P(1,1|3,3)&=\begin{pmatrix} 0 & 0 & 1 \\ 0 & 1 & 0 \\ 1 & 0 & 0 \end{pmatrix} \label{thirdobservables} \\
P_A(2|1)-P_A(1|1)=P_B(2|1)-P_B(1|1)&=\begin{pmatrix} 1 & 0 & 0 \\ 0 & -1 & 0 \\ 0 & 0 & -1 \end{pmatrix}. \label{marginals of first observables}
\end{align}
Moreover, there is no element in $C_q^{(3)}(3,2)$ satisfying equations (\ref{secondobservables})--(\ref{marginals of first observables}).
\end{theorem}

\begin{proof}
Let $\cH=\ell^2(\bZ)$, and define unit vectors in $\cH \otimes \cH$ by
\begin{equation}
\zeta_1=\sum_{j<0} (\sqrt{2})^j e_j \otimes e_j, \, \, \, \, \, \zeta_2=e_0 \otimes e_0, \, \, \, \, \, \zeta_3=\sum_{j>0} (\sqrt{2})^{-j} e_j \otimes e_j. \label{32vecs}
\end{equation}
Let $\{e_1,e_2,e_3\}$ denote the canonical orthonormal basis for $\bC^3$. It is easy to see that $\{ \zeta_1,\zeta_2,\zeta_3\}$ is an orthonormal set, so that the map $W:\bC^3 \to \cH \otimes \cH$ given by $We_i=\zeta_i$ for $i=1,2,3$ is an isometry. We define self-adjoint unitaries on the canonical basis vectors $\{e_j\}_{j \in \bZ}$ of $\ell^2(\bZ)$ by
\begin{align}
S_1 e_j&=\begin{cases} -e_j & j \geq 0 \\ e_j & j<0 \end{cases} \\
S_2e_j&=e_{-j+1} \\
S_3 e_j&=e_{-j}
\end{align}
It is straightforward to check that each $S_i$ is unitary and that $S_i^2=I$, so that each $S_i$ is a self-adjoint unitary. Thus, for each $x=1,2,3$, there are PVM's $\{E_{1,x},E_{2,x}\}$ on $\cH$ such that $S_x=E_{1,x}-E_{2,x}$. We let $F_{b,y}=E_{b,y}$ for all $y,b$, and we define $P=(W^*(E_{a,x} \otimes F_{b,y})W)_{a,b,x,y}$, which is an element of $C_{qs}^{(3)}(3,2)$. We note that for each $x=1,2,3$,
\begin{equation}
P(2,2|x,x)-P(2,1|x,x)-P(1,2|x,x)+P(1,1|x,x)=W^*(S_x \otimes S_x)W. \label{marginal in terms of observables}
\end{equation}
On the other hand,
\begin{equation}
P_A(2|1)-P_A(1|1)=W^*(S_1 \otimes I)W \text{ and } P_B(2|1)-P_B(1|1)=W^*(I \otimes S_1)W.
\end{equation}
Using the unitaries defined above, it is routine to check that equations (\ref{secondobservables})--(\ref{marginals of first observables}) are satisfied.

Now, suppose that $\widetilde{P} \in C_q^{(3)}(3,2)$ satisfies equations (\ref{secondobservables})--(\ref{marginals of first observables}). Then there are finite-dimensional Hilbert spaces $\cK_A$ and $\cK_B$, PVM's $\{\widetilde{E}_{1,x},\widetilde{E}_{2,x}\}$ on $\cK_A$ for each $x=1,2,3$, PVM's $\{\widetilde{F}_{1,y},\widetilde{F}_{2,y}\}$ on $\cK_B$ for each $y=1,2,3$, and an isometry $V:\bC^3 \to \cK_A \otimes \cK_B$ such that 
$$\widetilde{P}=(V^*(\widetilde{E}_{a,x} \otimes \widetilde{F}_{b,y})V)_{a,b,x,y}.$$
Since $\widetilde{E}_{1,x}-\widetilde{E}_{2,x}$ is unitary, there is a unital $*$-homomorphism $\pi_A:C^*(*_3 \bZ_2) \to \cB(\cK_A)$ such that $\pi_A(g_x)=\widetilde{E}_{1,x}-\widetilde{E}_{2,x}$, where $g_x$ is the generator of the $x$-th copy of $\bZ_2$ in $*_3 \bZ_2$. Similarly, there is a unital $*$-homomorphism $\pi_B:C^*(*_3 \bZ_2) \to \cB(\cK_B)$ such that $\pi_B(g_x)=\widetilde{F}_{1,y}-\widetilde{F}_{2,y}$.
Set $\eta_i=Ve_i$ for $i=1,2,3$, so that $\{ \eta_1,\eta_2,\eta_3\}$ is orthonormal. Considering the $(3,1)$-entry of equation (\ref{thirdobservables}), we have 
\begin{equation}
\la (\pi_A(g_3) \otimes \pi_B(g_3))\eta_1,\eta_3 \ra=1. \label{third unitary sends eta1 to eta3}
\end{equation}
Using the Cauchy-Schwarz inequality on equation (\ref{third unitary sends eta1 to eta3}), we have
\begin{equation}
\pi_A(g_3) \otimes \pi_B(g_3)\eta_1=\eta_3. \label{g3 actually sends eta1 to eta3}
\end{equation}
Considering the $(1,3)$ and $(2,3)$ entries of equation (\ref{secondobservables}), we see that
\begin{equation}
\la \pi_A(g_2) \otimes \pi_B(g_2)\eta_3,\eta_1 \ra=\la \pi_A(g_2) \otimes \pi_B(g_2)\eta_3,\eta_2 \ra=\frac{1}{\sqrt{2}}. \label{g2 inner products are one over root 2}
\end{equation}
Now, let $T_1=\pi_A(g_1) \otimes I$ and $T_2=I \otimes \pi_B(g_1)$, and set $U_1=\pi_A(g_2)\pi_A(g_3)=\pi_A(g_2g_3)$ and $U_2=\pi_B(g_2)\pi_B(g_3)=\pi_B(g_2g_3)$. Combining equations (\ref{g3 actually sends eta1 to eta3}) and (\ref{g2 inner products are one over root 2}), it follows that
$$\la (U_1 \otimes U_2)\eta_1,\eta_2 \ra=\la (U_1 \otimes U_2)\eta_1,\eta_1 \ra=\frac{1}{\sqrt{2}}.$$
Considering the $(1,1)$ and $(2,2)$ entries of equation (\ref{marginals of first observables}), we also have
$$\la (T_1 \otimes I)\eta_1,\eta_1 \ra=\la (I \otimes T_2)\eta_1,\eta_1 \ra=1$$
and
$$\la (T_1 \otimes I)\eta_2,\eta_2 \ra=\la(I \otimes T_2) \eta_2,\eta_2 \ra=-1.$$
Therefore, the unitaries $T_1,T_2,U_1,U_2$ and the unit vectors $\eta_1,\eta_2$ satisfy equations (\ref{GHJspatialnonfiniteexample})--(\ref{GHJmarginals}) of Theorem \ref{theorem: spatial infinite-dimensional unitaries}, contradicting the assumption that $\cK_A$ and $\cK_B$ are finite dimensional.
\end{proof}

\section{Two Inputs and Three Outputs}

In this section, we will show that $C_q^{(n)}(2,3) \neq C_{qs}^{(n)}(2,3)$ for some $n \in \{2,3,4\}$. Since elements of $C_{qs}^{(n)}(2,3)$ arise from tensor products of representations of $C^*(\bZ_3*\bZ_3)$, we aim to transform the representation of $C^*(\bZ_2*\bZ)$ from Proposition \ref{proposition: spatial rep of tensor of z2*z} into some representation of $C^*(\bZ_3*\bZ_3)$. However, we obtain a simpler group embedding (and hence a smaller number of required unit vectors) by first considering the group $\bZ_2*\bZ_3$. The following fact has a simple proof; we include it for convenience.

\begin{proposition}
\label{proposition: embedding of z2*z into z2*z3}
Let $g$ be the generator of $\bZ_2$; let $h$ be a generator of $\bZ_3$; and let $u$ be a generator of $\bZ$. Then there is an injective group homomorphism $\iota:\bZ_2*\bZ \hookrightarrow \bZ_2*\bZ_3$ such that $\iota(g)=g$ and $\iota(u)=hgh$.
\end{proposition}

\begin{proof}
Note that $hgh$ has infinite order since $h$ is order $3$ and powers of $hgh$ do not decrease in word length.  On the other hand, let
$$X_1=\{ w \in \bZ_2*\bZ_3: w \text{ starts with either } h \text{ or } h^2\}$$
and
$$X_2=\{w \in \bZ_2*\bZ_3: w \text{ starts with } g\},$$
where we consider words in reduced form. Clearly $gX_1 \subseteq X_2$ and $hX_2 \subseteq X_1$.  Thus, $hghX_2 \subseteq hgX_1 \subseteq hX_2 \subseteq X_1$. By the Ping-Pong Lemma (see \cite{delaharpe}), the map $\iota:\bZ_2*\bZ \to \bZ_2*\bZ_3$ given by $\iota(g)=g$ and $\iota(u)=hgh$ extends to an injective group homomorphism.
\end{proof}

Using the group embedding from Proposition \ref{proposition: embedding of z2*z into z2*z3}, we can translate the unitaries $T$ and $U$ from equation \ref{GHJunitaries} and the unit vectors $\zeta_1$ and $\zeta_2$ from equation \ref{GHJvectors} to tensor product representations of $C^*(\bZ_2*\bZ_3) \otimes C^*(\bZ_2*\bZ_3)$ that cannot be witnessed by tensor products of finite-dimensional representations. In this case, we only need to specify certain equations governing the representations and the unit vectors using words of length at most three.

\begin{lemma}
\label{lemma: spatial nonfinite for z2*z3}
There is an infinite-dimensional Hilbert space $\cH$, a unital $*$-homomorphism $\sigma:C^*(\bZ_2*\bZ_3) \to \bofh$, and unit vectors $\zeta_1,\zeta_2 \in \cH \otimes \cH$ such that
\begin{align}
\la(\sigma(hgh) \otimes \sigma(hgh))\zeta_1,\zeta_1 \ra&=\la(\sigma(hgh) \otimes \sigma(hgh))\zeta_1,\zeta_2 \ra=\frac{1}{\sqrt{2}}, \label{bab on GHJvectors}\\
\la (\sigma(g) \otimes I)\zeta_1,\zeta_1 \ra&=\la(I \otimes \sigma(g))\zeta_1,\zeta_1 \ra=1, \label{marginals on GHJvector1} \\
\la (\sigma(g) \otimes I)\zeta_2,\zeta_2 \ra&=\la(I \otimes \sigma(g))\zeta_2,\zeta_2 \ra=-1. \label{marginals on GHJvector2}
\end{align}
Moreover, these equations cannot be witnessed by a tensor product of finite-dimensional representations of $C^*(\bZ_2*\bZ_3)$.
\end{lemma}

\begin{proof}
Let $\cS=\spn \{1,g,hgh,h^2gh^2\} \subseteq C^*(\bZ_2*\bZ_3)$. Since $g^*=g$ and $(hgh)^*=h^*gh^*=h^2gh^2$, $\cS$ is an operator system. The group embedding $\bZ_2*\bZ \hookrightarrow \bZ_2*\bZ_3$ from Proposition \ref{proposition: embedding of z2*z into z2*z3} induces an injective $*$-homomorphism $C^*(\bZ_2*\bZ) \hookrightarrow C^*(\bZ_2*\bZ_3)$, with a completely positive expectation onto the range \cite[Proposition 8.8]{pisierbook}. Restricting to $\cS$, we obtain a complete order isomorphism of the operator system $\cP=\spn \{1,g,u,u^*\} \subseteq C^*(\bZ_2*\bZ)$ onto $\cS$ via $g \mapsto g$ and $u \mapsto hgh$. Let $\pi:C^*(\bZ_2*\bZ) \to \cB(\ell^2(\bZ))$ be the unital $*$-homomorphism given by $\pi(g)=T$ and $\pi(u)=U$, where $T,U$ are the operators in Proposition \ref{proposition: spatial rep of tensor of z2*z}. Since $\cS \simeq \cP$, the restriction of $\pi$ to $\cS$ gives a unital completely positive map $\psi:\cS \to \cB(\ell^2(\bZ))$. By Arveson's extension theorem \cite{arveson69}, we may extend $\psi$ to a unital completely positive map $\varphi:C^*(\bZ_2*\bZ_3) \to \cB(\ell^2(\bZ))$. Using Stinespring's dilation theorem \cite{stinespring}, there is a Hilbert space $\cH$, a unital $*$-homomorphism $\sigma:C^*(\bZ_2*\bZ_3) \to \cB(\cH)$ and an isometry $V:\ell^2(\bZ) \to \cH$ such that $\varphi(\cdot)=V^*\sigma(\cdot)V$. Since $V$ is an isometry, we identify $\ell^2(\bZ)$ with $V\ell^2(\bZ)$ and write $\cH=\ell^2(\bZ) \oplus \ell^2(\bZ)^{\perp}$.

Since $\ell^2(\bZ) \otimes \ell^2(\bZ) \subseteq \cH \otimes \cH$, we may identify $\zeta_1=\sum_{j<0} (\sqrt{2}) e_j \otimes e_j$ and $\zeta_2=e_0 \otimes e_0 \in \ell^2(\bZ) \otimes \ell^2(\bZ)$ as unit vectors in $\cH \otimes \cH$. It is not hard to check that equations (\ref{bab on GHJvectors})--(\ref{marginals on GHJvector2}) are satisfied.

Suppose that there are unital $*$-homomorphisms $\pi_A:C^*(\bZ_2*\bZ_3) \to \cB(\cH_A)$ and $\pi_B:C^*(\bZ_2*\bZ_3) \to \cB(\cH_B)$, along with unit vectors $\zeta_1,\zeta_2 \in \cH_A \otimes \cH_B$ satisfying equations (\ref{bab on GHJvectors})--(\ref{marginals on GHJvector2}), where $\cH_A$ and $\cH_B$ are finite-dimensional. Then setting $T_1=\pi_A(g)$, $T_2=\pi_B(g)$, $U_1=\pi_A(hgh)$ and $U_2=\pi_B(hgh)$, we would yield equations (\ref{GHJspatialnonfiniteexample})--(\ref{GHJmarginals}) on a tensor product of finite-dimensional Hilbert spaces, contradicting Theorem \ref{theorem: spatial infinite-dimensional unitaries}.
\end{proof}

We are now in a position to show that $C_q^{(4)}(2,3) \neq C_{qs}^{(4)}(2,3)$. For convenience, for $n \geq 2$, we will let $Q_n:\bC^2 \to \bC^n$ be the isometry sending $\bC^2$ to the first coordinates of $\bC^n$. We also let $\omega=\exp \left( \frac{2\pi i}{3} \right)$.

\begin{theorem}
\label{theorem: 2 inputs 3 outputs}
There exist $n \in \{2,3,4\}$, contractions $A$ and $B$ in $M_n$, and an element $P=P(a,b|x,y) \in C_{qs}^{(n)}(2,3)$ such that
\begin{align}
P(2,2|1,1)-P(1,2|2,2)-P(2,1|2,2)+P(1,1|2,2)&=A \label{matrixA}, \\
P(3,b|1,y)=P(a,3|x,1)&=0, \forall a,b,x,y \label{nothirdoutput} \\
\sum_{a,b=1}^3 \omega^{a+b} P(a,b|2,2)&=B, \label{matrixB} \\
Q_n^*(P_A(2|1)-P_A(1|1))Q_n=Q_n^*(P_B(2|1)-P_B(1|1))Q_n&=\begin{pmatrix} 1 & 0 \\ 0 & -1 \end{pmatrix}, \label{marginalcompressions} \\
\sum_{i=1}^n |B_{i1}|^2&=1, \label{first column of B is a unit vector} \\
\sum_{i=1}^n |(AB)_{i1}|^2&=1, \label{first column of AB is a unit vector} \\
(BAB)_{1,1}=(BAB)_{2,1}&=\frac{1}{\sqrt{2}}. \label{firstcolumnofBAB}
\end{align}
Moreover, for these contractions $A$ and $B$, if $\widetilde{P} \in C_{qs}^{(n)}(2,3)$ satisfies equations (\ref{matrixA})--(\ref{firstcolumnofBAB}), then $\widetilde{P} \not\in C_q^{(n)}(2,3)$.
\end{theorem}

\begin{proof}
Let $\sigma:C^*(\bZ_2*\bZ_3) \to \bofh$ and $\zeta_1,\zeta_2 \in \cH \otimes \cH$ be as in Lemma \ref{lemma: spatial nonfinite for z2*z3}. Then $\sigma_1,\sigma_2$ and $\zeta_1,\zeta_2$ satisfy equations (\ref{bab on GHJvectors})--(\ref{marginals on GHJvector2}). We define intermediate unit vectors
\begin{align}
\xi_3&=(\sigma(h) \otimes \sigma(h))\zeta_1, \label{xi3}\\
\xi_4&=(\sigma(g) \otimes \sigma(g))\xi_3=(\sigma(gh) \otimes \sigma(gh))\zeta_1. \label{xi4}
\end{align}
Although $\{ \zeta_1,\zeta_2\}$ is orthonormal, the set $\{ \zeta_1,\zeta_2,\xi_3,\xi_4\}$ may not be orthonormal. Applying the Gram-Schmidt orthonormalization process, we obtain an orthonormal basis $\zeta_1,...,\zeta_n$ for the subspace $\cM=\spn \{ \zeta_1,\zeta_2,\xi_3,\xi_4\}$ of $\cH \otimes \cH$, where $n=\dim(\cM) \in \{2,3,4\}$. Define $W:\bC^n \to \cH \otimes \cH$ by $We_i=\zeta_i$ for $1 \leq i \leq n$; then $W$ is an isometry. Since $\sigma(g)$ is a self-adjoint unitary, we may write $\sigma(g)=E_{2,1}-E_{1,1}$ for a PVM $\{E_{1,1},E_{2,1}\}$ on $\cH$. We extend this PVM to have three outputs by setting $E_{3,1}=0$. Since $\sigma(h)$ is an order three unitary, there is a PVM $\{E_{1,2},E_{2,2},E_{3,2}\}$ on $\cH$ such that $\sigma(h)=\sum_{b=1}^3 \omega^b E_{b,2}$. Define $F_{b,y}=E_{b,y}$ for all $y=1,2$ and $b=1,2,3$, and set
\begin{equation}
P(a,b|x,y)=W^*(E_{a,x} \otimes F_{b,y})W.
\end{equation}
Then $P=(P(a,b|x,y))_{a,b,x,y}$ defines an element of $C_{qs}^{(n)}(2,3)$. Putting together equations (\ref{xi3}) and (\ref{xi4}), we recover the equation
\begin{align}
(\sigma(h) \otimes \sigma(h))\xi_4=(\sigma(hgh) \otimes \sigma(hgh))\zeta_1=\frac{1}{\sqrt{2}}(\zeta_1+\zeta_2). \label{b otimes b on xi4 in terms of zetas}
\end{align}
Define elements $A,B,C,D$ of $M_n$ by
\begin{align}
A&=W^*(\sigma_1(a) \otimes \sigma_2(a))W, \, \, \, \, \, B=W^*(\sigma_1(b) \otimes \sigma_2(b))W  \label{matrices A and B}\\
C&=W^*(\sigma_1(a) \otimes I)W, \, \, \, \, \, \, \, \, \, \, \, \, \, \, 
D=W^*(I \otimes \sigma_2(a))W. \label{matrices C and D}
\end{align}
Then $A,B,C$ and $D$ are contractions. Moreover, since $$\sigma(g) \otimes \sigma(g)=E_{2,1} \otimes F_{2,1}-E_{1,1} \otimes F_{2,1}-E_{2,1} \otimes F_{1,1}+E_{1,1} \otimes F_{1,1},$$
the correlation $P$ satisfies equation (\ref{matrixA}). Similarly, considering $\sigma(h) \otimes \sigma(h)$, $P$ must satisfy equation (\ref{matrixB}). Equation (\ref{nothirdoutput}) follows since $E_{3,1}=F_{3,1}=0$. By the choice of the unitary $\sigma(a)$ and the unit vectors $\zeta_1$ and $\zeta_2$, equations (\ref{marginals on GHJvector1}) and (\ref{marginals on GHJvector2}) show that
\begin{equation}
Q_n^*CQ_n=Q_n^*DQ_n=\begin{pmatrix} 1 & 0 \\ 0 & -1 \end{pmatrix}. \label{upper 2 by 2 corner of C and D}
\end{equation}
Since $C=P_A(2|1)-P_A(1|1)$ and $D=P_B(2|1)-P_B(1|1)$, we obtain equation (\ref{marginalcompressions}).
Combining equations (\ref{xi3}), (\ref{matrices A and B}), for each $1 \leq i \leq 4$,
\begin{equation}
\langle \xi_3,\zeta_i \rangle=\langle (\sigma(h) \otimes \sigma(h))\zeta_1,\zeta_i \rangle=B_{i1}. \label{xi3 inner products are first column of B}
\end{equation}
Since $\xi_3 \in \cM$, it follows that
\begin{equation}
\xi_3=\sum_{j=1}^n B_{j1} \zeta_i. \label{xi3 in terms of B}
\end{equation}
Moreover, since $\{ \zeta_1,...,\zeta_n\}$ is an orthonormal basis for $\cM$ and $\xi_3$ is a unit vector, we have
\begin{equation}
1=\|\xi_3\|^2=\sum_{j=1}^n |B_{j1}|^2.
\end{equation}  
Thus, $B$ satisfies equation (\ref{first column of B is a unit vector}). Using equation (\ref{xi4}) and applying $\sigma(g) \otimes \sigma(g)$ to equation (\ref{xi3 in terms of B}), we obtain
\begin{align}
\langle \xi_4,\zeta_i \rangle&=\left\la (\sigma(g) \otimes \sigma(g)) \left( \sum_{j=1}^n B_{j1} \right) \zeta_j,\zeta_i \right\ra \\
&= \sum_{j=1}^n B_{j1} \langle (\sigma(g) \otimes \sigma(g))\zeta_j,\zeta_i \rangle \\
&=\sum_{j=1}^n A_{ij}B_{j1}=(AB)_{i1},
\end{align}
where $(AB)_{i1}$ denotes the $(i,1)$-entry of $AB$. Using the fact that $\xi_4 \in \cM$, we obtain
\begin{equation}
\xi_4=\sum_{i=1}^n (AB)_{i1} \zeta_j. \label{xi4 in terms of zetas}
\end{equation}
Since $\xi_4$ is a unit vector, the first column of $AB$ has norm $1$, which yields equation (\ref{first column of AB is a unit vector}).  An analogous argument with equation (\ref{b otimes b on xi4 in terms of zetas}) demonstrates that
\begin{equation}
\frac{1}{\sqrt{2}}(\zeta_1+\zeta_2)=\sum_{i=1}^n (BAB)_{i,1} \zeta_i,
\end{equation}
which forces equation (\ref{firstcolumnofBAB}) to be satisfied. We conclude that $P$, $A$ and $B$ satisfy all of equations (\ref{matrixA})--(\ref{firstcolumnofBAB}).
 
Now, suppose for a contradiction that there is $\widetilde{P}$ in $C_q^{(n)}(2,3)$ satisfying equations (\ref{matrixA})--(\ref{firstcolumnofBAB}). Then there are finite-dimensional Hilbert spaces $\cK_A$ and $\cK_B$, PVMs $\{ \widetilde{E}_{a,x}\}_{a=1}^3$ on $\cK_A$ for $x=1,2$, and PVMs $\{\widetilde{F}_{b,y}\}_{b=1}^3$ on $\cK_B$ for $y=1,2$, along with an isometry $V:\bC^n \to \cK_A \otimes \cK_B$ such that
$$\widetilde{P}=( V^*( \widetilde{E}_{a,x} \otimes \widetilde{F}_{b,y})V )_{a,b,x,y} \in C_q^{(n)}(2,3).$$
By equation (\ref{nothirdoutput}), $\widetilde{P}(3,b|1,y)=\widetilde{P}(a,3|x,1)=0$ for all $a,b,x,y$. By replacing $\widetilde{E}_{2,1}$ with $\widetilde{E}_{2,1}+\widetilde{E}_{3,1}$ if necessary, we may assume without loss of generality that $\widetilde{E}_{3,1}=0$, and that $\{ \widetilde{E}_{1,1},\widetilde{E}_{2,1}\}$ is a PVM.  Similarly, we may assume that $\widetilde{F}_{3,1}=0$, and that $\{ \widetilde{F}_{1,1},\widetilde{F}_{2,1}\}$ is a PVM. Since $\widetilde{E}_{2,1}-\widetilde{E}_{1,1}$ is a self-adjoint unitary and $\sum_{a=1}^3 \omega^a \widetilde{E}_{a,2}$ is an order three unitary, the map $\gamma_A:C^*(\bZ_2*\bZ_3) \to \cB(\cK_A)$ given by $\gamma_A(g)=\widetilde{E}_{2,1}-\widetilde{E}_{1,1}$ and $\gamma_A(h)=\sum_{a=1}^3 \omega^a \widetilde{E}_{a,2}$ extends to a unital $*$-homomorphism. Similarly, there is a unital $*$-homomorphism $\gamma_B:C^*(\bZ_2*\bZ_3) \to \cB(\cK_B)$ such that $\gamma_B(g)=\widetilde{F}_{2,1}-\widetilde{F}_{1,1}$ and $\gamma_B(h)=\sum_{b=1}^3 \omega^b \widetilde{F}_{b,2}$. Since $\widetilde{P}$ satisfies equations (\ref{matrixA}), (\ref{matrixB}) and (\ref{marginalcompressions}), it follows that
\begin{align}
V^*(\gamma_A(g) \otimes \gamma_B(g))V&=A, \, \, \, V^*(\gamma_A(h) \otimes \gamma_B(h))V=B, \text{ and} \\
Q_n^* V^*(I_{\cK_A} \otimes \gamma_B(g))VQ_n&=\begin{pmatrix} 1 & 0 \\ 0 & -1 \end{pmatrix}=Q_n^* V^*(\gamma_A(g) \otimes I_{\cK_B})VQ_n. \label{compression of f.d. marginals}
\end{align}
For $1 \leq i \leq n$, define $\eta_i=Ve_i$. Then $\{\eta_1,...,\eta_n\}$ is an orthonormal set. We define vectors
\begin{align}
\chi_3&=\sum_{i=1}^n B_{i1} \eta_i, \text{ and} \\
\chi_4&=\sum_{i=1}^n (AB)_{i1} \eta_i.
\end{align}
These are unit vectors by equations (\ref{first column of B is a unit vector}) and (\ref{first column of AB is a unit vector}). Applying the Cauchy-Schwarz inequality and noting that $B_{i1}=\la (\gamma_A(h) \otimes \gamma_B(h))\eta_1,\eta_i \ra$, it readily follows that
\begin{align}
(\gamma_A(h) \otimes \gamma_B(h))\eta_1&=\chi_3, \text{ and } \\
(\gamma_A(g) \otimes \gamma_A(g))\chi_3&=\chi_4.
\end{align}
Using equation (\ref{firstcolumnofBAB}) and applying Cauchy-Schwarz again, it follows that
\begin{equation}
(\gamma_A(hgh) \otimes \gamma_B(hgh))\eta_1=(\gamma_A(h) \otimes \gamma_B(h))\chi_4=\frac{1}{\sqrt{2}}(\eta_1+\eta_2). \label{bab f.d. contradiction}
\end{equation}
A similar argument using equation (\ref{compression of f.d. marginals}) demonstrates that
\begin{align}
(\gamma_A(g) \otimes I)\eta_1&=(I \otimes \gamma_B(g))\eta_1=\eta_1, \label{marginals eta_1 f.d. contradiction} \\
(\gamma_A(g) \otimes I)\eta_2&=(I \otimes \gamma_B(g))\eta_2=-\eta_2. \label{marginals eta_2 f.d. contradiction}
\end{align}
Thus, combining equations (\ref{bab f.d. contradiction})--(\ref{marginals eta_2 f.d. contradiction}), we can realize equations (\ref{bab on GHJvectors})--(\ref{marginals on GHJvector2}) in a finite-dimensional tensor product setting, which contradicts Lemma \ref{lemma: spatial nonfinite for z2*z3}. Therefore, we obtain the separation $C_q^{(n)}(2,3) \neq C_{qs}^{(n)}(2,3)$, as desired.
\end{proof}

Combining Theorems \ref{theorem: 3 inputs 2 outputs} and \ref{theorem: 2 inputs 3 outputs} shows that for any $(m,k)$ with $m,k \geq 2$ and $(m,k) \neq (2,2)$, we have $C_q^{(n)}(m,k) \neq C_{qs}^{(n)}(m,k)$ for some matrix level $n$, with $n \leq 4$. This result is optimal with respect to the input and output sets. Indeed, if $m=k=2$, then we have
$$C_q^{(n)}(2,2)=C_{qs}^{(n)}(2,2)=C_{qa}^{(n)}(2,2),$$
since the underlying group, $\bZ_2*\bZ_2$, is amenable and has the property that every irreducible representation is at most $2$-dimensional. On the other hand, while it is still unknown whether $C_q(3,2) \neq C_{qs}(3,2)$ or $C_q(2,3) \neq C_{qs}(2,3)$, Theorem \ref{theorem: matrix version of q always different from qs} provides some partial evidence that these separations may possibly hold.

\section*{Acknowledgements}

This research was conducted during a visit to the University of Copenhagen. The author would like to thank the university for their kind hospitality.  We thank Laura Man\v{c}inska for many helpful insights and feedback. We also thank Li Gao and Marius Junge for their valuable comments.

\end{document}